\documentclass[letterpaper, 12pt]{article}
\usepackage{amsmath,amsthm,amssymb,enumerate,mathscinet}
\usepackage{fullpage}
\newtheorem{theorem}{Theorem}
\newtheorem{corollary}[theorem]{Corollary}
\newtheorem{prop}[theorem]{Proposition}
\newtheorem{lemma}[theorem]{Lemma} 
\def\A{\mathcal{A}}

\def\P{\mathcal{P}}
\def\eps{\varepsilon}
\title{A random version of Sperner's theorem}

\author{J\'ozsef Balogh,\footnote{Department of Mathematics, University of Illinois, Urbana, IL 61801, USA and Bolyai Institute, University of Szeged, Szeged, Hungary {\tt jobal@math.uiuc.edu}.
   Research is partially supported by Simons Fellowship, NSF CAREER Grant DMS-0745185, Arnold O. Beckman Research Award (UIUC Campus Research Board 13039) and Marie Curie FP7-PEOPLE-2012-IIF 327763.} 
\ Richard Mycroft\footnote{University of Birmingham, United Kingdom, {\tt r.mycroft@bham.ac.uk}.}
\  and Andrew Treglown\footnote{University of Birmingham, United Kingdom, {\tt a.c.treglown@bham.ac.uk}.}}
\date{}

\begin{document}
 \maketitle
\begin{abstract}
Let $\P(n)$ denote the power set of $[n]$, ordered by inclusion, and let $\P (n,p)$ be obtained from $\P(n)$ by selecting elements from $\P (n)$ independently at random with probability $p$. A classical result of Sperner~\cite{sperner} asserts that every antichain in $\P(n)$ has size at most that of the middle layer, $\binom{n}{\lfloor n/2 \rfloor}$. In this note we prove an analogous result for $\P (n,p)$: If $pn \rightarrow \infty$ then, with high probability, the size of the largest antichain in $\P(n,p)$ is at most $(1+o(1)) p \binom{n}{\lfloor n/2 \rfloor}$. This solves a conjecture of Osthus~\cite{osthus} who proved the result in the case when $pn/\log n \rightarrow \infty$. Our condition on $p$ is best-possible. In fact, we prove a more general result giving an upper bound on the size of the largest antichain for a wider range of values of $p$.
\end{abstract}

We write $[n]$ for the set of natural numbers up to $n$, and $\P(n)$ for the power set of $[n]$. Also, for any $0 \leq k \leq n$ we write $\binom{[n]}{k}$ for the subset of $\P(n)$ consisting of all sets of size $k$. A subset $\A \subseteq \P(n)$ is an \emph{antichain} if for any $A, B \in \A$ with $A \subseteq B$ we have $A = B$. So $\binom{[n]}{k}$ is an antichain for any $0 \leq k \leq n$; Sperner's theorem~\cite{sperner} states that in fact no antichain in $\P(n)$ has size larger than $\binom{n}{\lfloor n/2 \rfloor}$. Our main theorem is a random version of Sperner's theorem. For this, let $\P(n, p)$ be the set obtained from $\P(n)$ by selecting elements randomly with probability $p$ and independently of all other choices. Write $m := \binom{n}{\lfloor n/2 \rfloor}$. Roughly speaking, our main result asserts that if $p> C/n$ for some constant $C$, then with high probability, the largest antichain in $\P(n,p)$ is approximately the same size as the `middle layer' in $\P(n,p)$.

\begin{theorem} \label{main}
For any $\eps > 0$ there exists a constant $C$ such that if $p > C/n$ then with high probability the largest antichain in $\P(n,p)$ has size at most $(1+\eps)pm$.
\end{theorem}
(Here, by `with high probability' we mean with probability tending to $1$ as $n$ tends to infinity.) 

The model $\P(n,p)$ was first investigated by R\'enyi~\cite{renyi} who determined the probability threshold for the property that $\P(n,p)$ is not itself an antichain, thereby answering a question of Erd\H{o}s. 
 The size of the largest antichain in $\P(n,p)$ for $p$ above this threshold was first studied by Kohayakawa and Kreuter~\cite{kk}. In~\cite{kk} they raised the question of which values of $p$ does the conclusion of Theorem~\ref{main} hold.
Osthus~\cite{osthus} proved Theorem~\ref{main} in the case when $pn/\log n \rightarrow \infty$ and conjectured that this can be replaced by $pn \rightarrow \infty$. (So Theorem~\ref{main} resolves this conjecture.) Moreover, Osthus showed that, for a fixed $c>0$, if $p=c/n$ then with high probability the largest antichain in $\P(n,p)$ has size at least $(1+o(1))(1+e^{-c/2}) p \binom{n}{\lfloor n/2 \rfloor}$. So the bound on $p$ in Theorem~\ref{main} is best-possible up to the constant $C$. There have also been a number of results concerning the length of (the longest) chains in $\P(n,p)$ and related models of random posets (see
for example, \cite{bol, koh, kre}).

Instead of proving Theorem~\ref{main} directly we prove the following more general result.
\begin{theorem} \label{main2}
Let $n \in \mathbb N$ and $m:=\binom{n}{\lfloor n/2 \rfloor}$.
For any $\eps > 0$  and $t \in \mathbb N$, there exists a constant $C$ such that if $p > C/n^t$ then with high probability the largest antichain in $\P(n,p)$ has size at most $(1+\eps)pmt$.
\end{theorem}
Osthus~\cite{osthus} proved this result in the case when $p(n/t)^t /\log n \rightarrow \infty$. (In fact, Osthus's result allows for $t$ to be an integer function, see~\cite{osthus} for the precise statement.) Moreover, Osthus showed that, for 
$1/n^{t} \ll p \ll 1/n^{t-1}$, with high probability, $\P(n,p)$ has an antichain of size at least $(1+o(1))pmt$ (so Theorem~\ref{main2} is `tight' in this window of $p$).

The method of proof of Theorem~\ref{main2} also allows us to estimate the number of antichains in $\P(n)$ of certain fixed sizes.

\begin{prop} \label{countantichains}
Fix any $t \in \mathbb{N}$, and suppose that $m/n^t \ll s \ll m/n^{t-1}$. Then the number of antichains of size $s$ in $\P(n)$ is $\binom{(t+o(1))m}{s}$.
\end{prop}

To prove Theorem~\ref{main2}, let $G$ be the graph with vertex set $\P(n)$ in which distinct sets $A$ and $B$ are adjacent if $A \subseteq B$ or $B \subseteq A$. Then an antichain in $\P(n)$ is precisely an independent set in $G$. We follow the `hypergraph container' approach (see, for example, \cite{container1, container2}): indeed, we show that all independent sets in $G$ are contained within a fairly small number of low-density sets in $G$. Crucially, for this method to work, we have to construct our `containers' in two phases (see Lemma~\ref{containers}).
For this we use a result of Kleitman~\cite{kleitman} on the minimum number of edges induced by a subset of $G$ with a given fixed size. Define the \emph{centrality order} on the vertices of $\P(n)$ as follows: we begin with the elements of $\binom{[n]}{\lfloor n/2 \rfloor}$, ordered arbitrarily, then the elements of $\binom{[n]}{\lfloor n/2 \rfloor + 1}$, then the elements of $\binom{[n]}{\lfloor n/2 \rfloor -1}$, then the elements of $\binom{[n]}{\lfloor n/2 \rfloor +2}$, and so forth until all vertices of $\P(n)$ have been ordered. For any $r \in \mathbb{N}$ let $I_r$ denote the initial segment of this order of length $r$; Kleitman~\cite{kleitman} proved that $I_r$ minimises the number of induced edges over all sets of size $r$ (see also~\cite{das}, which characterises all the sets $U$ of size $r$ for which $e(G[U])$ is minimised).

\begin{theorem}[Kleitman~\cite{kleitman}] \label{kleitman}
For any $r \leq 2^n$ and any $U \subseteq V(G)$ of size $r$ we have $e(G[U]) \geq e(G[I_r])$.
\end{theorem}

We apply this theorem in the form of the following corollary.

\begin{corollary} \label{kleitmancoro}
Let $U \subseteq V(G)$, and suppose that $0<\eps \leq 1/2$ and $t \in \mathbb N$. If $|U|\ge (t+\eps)m$, then $e(G[U])> \eps n^t |U|/(2t)^{t+1}$.
\end{corollary}

\begin{proof}
Let $r := |U|$. We have $r \geq (t+\eps)m$, so in particular $r-mt  \geq  r (1-t/(t+\eps )) \geq 2\eps r/(1+2t)$ since $\eps \leq 1/2$. Observe that $I_r$ contains all of the at most $mt$ elements of the $t$ `middle layers', $\binom{[n]}{\lfloor n/2 \rfloor}$, $\binom{[n]}{\lfloor n/2 \rfloor+1}$, and so forth. Further, $I_r$ contains at least $r-mt$ elements from outside these layers, each of which has at least $\binom{\lceil n/2 \rceil}{t} \geq (n/2t)^t$ neighbours in  the $t$ middle layers. So by Theorem~\ref{kleitman} we have 
\begin{equation*}
e(G[U]) \geq e(G[I_r]) \geq \frac{2\eps r}{1+2t} \cdot \left( \frac{n}{2t}\right)^t \geq \frac{\eps n^t r}{(2t)^{t+1}}. \qedhere
\end{equation*}
\end{proof}

Let $s \in \mathbb N$, $t >0$ and let $S$ be a set of size $|S|=s$. Define $\binom{S}{\leq t}$ to be the set of all subsets of $S$ of size at most $t$ and $\binom{s}{\leq t}:=\big |\binom{S}{\leq t} \big |$.

\begin{lemma} \label{containers} 
Suppose that $t \in \mathbb N$,  $0 < \eps \leq 1/(2t)^{t+1}$ and $n$ is sufficiently large. Then there exist functions $f : \binom{V(G)}{\leq n^{-(t+0.9)}2^n} \to \binom{V(G)}{\leq (t+1+\eps)m}$ and $g : \binom{V(G)}{\leq (t+2)m/(\eps ^2 n^t)} \to \binom{V(G)}{\leq (t+\eps)m}$ such that, for any independent set $I$ in $G$, there are disjoint subsets $S_1, S_2 \subseteq I$ with
$S_1 \in \binom{V(G)}{\leq n^{-(t+0.9)}2^n} $, $S_2 \in \binom{V(G)}{\leq (t+2)m/(\eps ^2 n^t)}$
such that  $S_1 \cup S_2$ and $g(S_1 \cup S_2)$ are disjoint, $S_2 \subseteq f(S_1)$, and $I \subseteq S_1 \cup S_2 \cup g(S_1 \cup S_2)$. 
\end{lemma}
Roughly speaking, Lemma~\ref{containers} ensures that every independent set $I$ in $G$ lies in some (not too big) sparse `container' set $S_1 \cup S_2 \cup g(S_1 \cup S_2)$, and in total we do not have `too many' containers. Indeed, since $S_1$ and $S_2$ are small sets, there are not too many possibilities for the set $S_1 \cup S_2$, which in turn means there are not too many containers $S_1 \cup S_2 \cup g(S_1 \cup S_2)$ to consider. This property is crucial to the proof of Theorem~\ref{main2}, as it enables us to take a union bound to show that it is unlikely that the number of vertices randomly selected from any container is significantly higher than expected. 
\begin{proof}[Proof of Lemma~\ref{containers}]
Fix an arbitrary total order $v_1, \dots, v_n$ on the vertices of $V(G)$. Given any independent set $I$ in $G$, define $G_0 := G$, and take $S_1$ and $S_2$ to be initially empty. We add vertices to $S_1$ and $S_2$  through the following iterative process, beginning at Step 1 in Phase 1.

\emph{Phase 1:} At Step $i$, let $u$ be the maximum degree vertex of $G_{i-1}$ (with ties broken by our fixed total order). If $u \notin I$ then define $G_i := G_{i-1} \setminus \{u\}$, and proceed to Step $i+1$ (still in Phase 1). Alternatively, if $u \in I$ and $\deg_{G_{i-1}}(u) \geq n^{t+0.9}$ then add $u$ to $S_1$, define $G_i := G_{i-1} \setminus (\{u\} \cup N_G(u))$, and proceed to Step $i+1$ (still in Phase 1). Finally, if $u \in I$ and $\deg_{G_{i-1}}(u) < n^{t+0.9}$, then add $u$ to $S_1$, define $G_i := G_{i-1} \setminus \{u\}$ and $f(S_1) := V(G_i)$, and proceed to Step $i+1$ of Phase 2.

\emph{Phase 2:} At Step $i$, let $u$ be the maximum degree vertex of $G_{i-1}$. If $u \notin I$ then define $G_i := G_{i-1} \setminus \{u\}$, and proceed to Step $i+1$ (still in Phase 2). Alternatively, if $u \in I$ and $\deg_{G_{i-1}}(u) \geq \eps ^2 n^t$ then add $u$ to $S_2$, define $G_i := G_{i-1} \setminus (\{u\} \cup N_G(u))$, and proceed to Step $i+1$ (still in Phase 2). Finally, if $u \in I$ and $\deg_{G_{i-1}}(u) < \eps ^2 n^t$, then add $u$ to $S_2$, define $G_i := G_{i-1} \setminus \{u\}$ and $g(S_1 \cup S_2):=V(G_i)$, and terminate.

\medskip

Observe first that for any independent set $I$ in $G$ the process defined ensures that $S_1$ and $S_2$ are disjoint subsets of $I$, that $S_1 \cup S_2$ is disjoint from $g(S_1 \cup S_2)$, that $S_2 \subseteq f(S_1)$ and that $I \subseteq S_1 \cup S_2 \cup g(S_1 \cup S_2)$. 

Next, note that for any independent set $I$, if a vertex $u$ is added to $S_1$ at step $i$, $u$ and at least $n^{t +0.9}$ neighbours of $u$ are deleted from $G_{i-1}$ in forming $G_i$, with a single exception (when $u$ is the final vertex added to $S_1$). So we must have $|S_1| \leq 1+ |V(G)|/(n^{t+0.9}+1) \leq n^{-(t+0.9)}2^n$. Furthermore, at the end of Phase 1 we know that every vertex $v$ of $G_i$ has $\deg_{G_i}(v) \leq n^{t+0.9}$, and so Corollary~\ref{kleitmancoro} implies that $f(S_1)$, the set of all vertices not deleted up to this point, must have size $|f(S_1)| < (t+1+\eps)m$. Then, in Phase 2, if a vertex $u$ is added to $S_2$ at step $i$, at least $\eps ^2 n^t$ neighbours of $u$ are deleted from $G_{i-1}$ in forming $G_i$, again with the single exception of the final vertex added to $S_2$. So we must have $|S_2| \leq 1+|f(S_1)|/(\eps ^2 n^t) $ and thus $$|S_1 \cup S_2| \leq 1+(t+1+\eps )m/(\eps ^2 n^t) + n^{-(t+0.9)}2^n \leq (t+2) m/(\eps ^2 n^t).$$ Moreover, at the end of Phase 2 every vertex $v$ of the final $G_i$ has $\deg_{G_i}(v) \leq \eps ^2 n ^t$ and so $e(G_i)\leq \eps ^2 n^t |G_i| \leq \eps n^t |G_i|/(2t)^{t+1}$. Thus,
Corollary \ref{kleitmancoro} implies that $|g(S_1 \cup S_2)| \leq (t+\eps)m$.

So it is sufficient to check that the functions $f$ and $g$ are well-defined. That is, we must check that if the process described above yields the same set $S_1$ when applied to independent sets $I$ and $I'$, then it should also yield the same set $f(S_1)$, and if additionally the same set $S_2$ is returned then the sets $g(S_1 \cup S_2)$ should be identical. However, this is a consequence of the fact that we always chose $u$ to be the vertex of $I$ of maximum degree in $G_{i-1}$. Moreover, if our algorithm produces sets $S_1 , S_2$ for an independent set $I$ and  sets $S'_1 , S'_2$ for an independent set $I'$ such that $S_1 \cup S_2 =
S'_1 \cup S'_2$ then $S_1=S'_1$ (and $S_2=S'_2$). 
 Thus, indeed $f$ and $g$ are well-defined.
\end{proof} 

The reason for using a two-phase algorithm in the proof of Lemma~\ref{containers} is that the structure of the hypercube graph is locally highly asymmetric; even worse, the size of the targeted independent set $I$ is very small compared to the number of vertices in the graph. Roughly speaking, the main objective of Phase 1 (where in each step many vertices are removed) is to decrease the number of potential vertices of $I$ sufficiently for the standard `hypergraph container' approach of Phase 2 to be successful.


\begin{proof}[Proof of Theorem~\ref{main2}]
Fix $\eps > 0$ and $t \in \mathbb N$; we may assume that $\eps < 1/(2t)^{t+1}$. Define $C:=10^{10} \eps ^{-5}$ and $\eps _1 := \eps/4$. Let $G_p$ be the graph formed from $G$ by selecting vertices independently at random with probability $p>C/n^t$. Then we must show that, with high probability, $G_p$ has no independent set of size greater than $(1+\eps)pmt$. 
Apply Lemma~\ref{containers} with $\eps _1$ playing the role of $\eps$.
Suppose for a contradiction that $G_p$ does contain some independent set $I$ with $|I| > (1+\eps)pmt$.
 Then all vertices of the sets $S_1$ and $S_2$ given by Lemma~\ref{containers} for this $I$ must have been selected for $G_p$, along with at least $|I| - |S_1 \cup S_2| \geq (1+\eps)pmt -  (t+2)m/(\eps _1 ^2 n^t)\geq (1+\eps/2)pmt$  vertices of $g(S_1 \cup S_2)$ (the second inequality follows from $C=10^{10} \eps ^{-5}$). 

However, the number of possibilities for $S_1$ is $\binom{2^n}{\leq n^{-(t+0.9)} 2^n}$, and for each possibility the probability that $S_1 \subseteq V(G_p)$ is $p^{|S_1|}$. For any fixed $S_1$ we have $|f(S_1)| \leq (t+2)m$ and $S_2 \subseteq f(S_1)$, so the number of possibilities for $S_2$ is at most $\binom{(t+2)m}{\leq (t+2)m/(\eps _1 ^2 n^t)}$, and for each possibility the probability that $S_2 \subseteq V(G_p)$ is $p^{|S_2|}$. Finally, for any fixed $S_1$ and $S_2$ we have $g(S_1 \cup S_2) \leq (t+\eps_1) m \leq (1+\eps/4)mt$, so the expected number of vertices of $g(S_1 \cup S_2)$ selected for $G_p$ is at most $(1+\eps/4)pmt$. By a standard Chernoff bound the probability that at least $(1+\eps/2)pmt$ vertices of $g(S_1 \cup S_2)$ are selected for $G_p$ is therefore at most $e^{-\eps^2pmt/100}$. Taking a union bound, we conclude that the probability that $G_p$ contains an independent set~$I$ of size greater than $(1+\eps) pmt
$ is at most  
\begin{align*}
\Pi &:= \sum _{0\leq a \leq n^{-(t+0.9)} 2^n}  \ \sum _{0\leq b \leq (t+2)m/(\eps _1 ^2 n^t) } \binom{2^n}{a}\cdot p^{a}\cdot \binom{(t+2)m}{b}\cdot p^{b} \cdot  
e^{-\eps^2 pmt/100} \\
&
\leq (n^{-(t+0.9)} 2^n +1) ((t+2)m/(\eps _1 ^2 n^t) +1) \binom{2^n}{n^{-(t+0.9)} 2^n} \cdot p^{n^{-(t+0.9)} 2^n}
\binom{(t+2)m}{(t+2)m/(\eps _1 ^2 n^t)}\\ & \ \ \cdot  p^{(t+2)m/(\eps _1 ^2 n^t)} \cdot 
e^{-\eps^2 pmt/100}.
\end{align*}
Note that for large $n$, with plenty of room to spare we have
$$(n^{-(t+0.9)} 2^n +1) ((t+2)m/(\eps _1 ^2 n^t) +1) \leq e^{\eps^2 pmt/400}$$ and
$$\binom{2^n}{n^{-(t+0.9)} 2^n} \cdot p^{n^{-(t+0.9)} 2^n} \leq e^{\eps^2 pmt/400}.$$
Further, since $C=10^{10} \eps ^{-5}$, for large $n$ we have that
$$\binom{(t+2)m}{(t+2)m/(\eps _1 ^2 n^t)} \cdot  p^{(t+2)m/(\eps _1 ^2 n^t)} \leq e^{\eps^2 pmt/400}.$$
Thus, the upper bound $\Pi$ on the probability is $o(1)$.
\end{proof} 

We conclude with a sketch of the proof of Proposition~\ref{countantichains}, on the number of antichains of given fixed sizes in $\P(n)$.

\begin{proof}[Proof sketch of Proposition~\ref{countantichains}]
The lower bound can be obtained by greedily choosing vertices from within the $t$ middle layers of $\P(n)$ to form an antichain of size $s$, and counting the number of ways to make these choices. For the upper bound, fix any $\eps > 0$ and apply Lemma~\ref{containers} with this $\eps$ and $t$. Then any independent set in $G$ of size $s$ is uniquely determined by the choice of 
\begin{enumerate}
\item a set $S_1$ of size $s_1 \leq \ell_1 := 2^n/n^{t+0.9}$, for which there are at most $\binom{2^n}{\leq \ell_1}$ choices,
\item a set $S_2 \subseteq f(S_1)$ of size $s_2 \leq \ell_2 := (t+2)m/(\eps^2n^t)$, for which there are at most $\binom{(t+1+\eps)m}{\leq \ell_2}$ choices, and
\item a set $S \subseteq g(S_1 \cup S_2)$ of size $s - s_1 - s_2$, for which there are at most $\binom{(t+\eps)m}{s-s_1-s_2}$ choices.
\end{enumerate}
Summing over all these choices by a similar calculation as in the proof of Theorem~\ref{main2}, we find that (for large $n$) there are at most $\binom{(t+2\eps)m}{s}$ independent sets of size $s$ in $G$.
\end{proof}

When we completed the project, we were informed that Collares Neto and Morris~\cite{cm} independently proved Theorem~\ref{main}. Their method is however different. We used the proof technique of \cite{container1}, and they followed the method of \cite{container2}. In particular, when we constructed containers, we aimed at having few vertices, whilst they aimed at having only few edges.

\section*{Acknowledgements}
The authors are grateful to the strategic alliance between the University of Birmingham and the University of Illinois at Urbana-Champaign. Much of the research for this paper was carried out during visits in connection with this partnership. The authors are also grateful to Deryk Osthus for a discussion on~\cite{osthus} and to the referees for their quick and careful reviews.

\end{document}